\documentclass[12pt,a4paper]{article}
\usepackage[utf8]{inputenc}
\usepackage[margin=1in]{geometry}
\usepackage{amsmath,amssymb,amsthm,amsfonts}
\usepackage{enumitem,booktabs,graphicx,tikz,longtable}
\usepackage[colorlinks=true,linkcolor=blue,citecolor=blue,urlcolor=blue]{hyperref}

\theoremstyle{definition}
\newtheorem{definition}{Definition}[section]

\theoremstyle{plain}
\newtheorem{theorem}[definition]{Theorem}

\theoremstyle{remark}

\usepackage[ruled,vlined]{algorithm2e}

\usepackage[T1]{fontenc}        

\title{\textbf{Ternary $\Gamma$-Semirings as a Novel Algebraic Framework
for Learnable Symbolic Reasoning}}

\author{\small
Chandrasekhar Gokavarapu$^{1,2}$\hspace{5mm} Dr D Madhusudhana Rao$^{3,4}$\\[10pt]
\small $^1$Lecturer in Mathematics,
Government College (A), Rajahmundry, A.P., India\\[2pt]
 \small $^2$Research Scholar, Department of Mathematics,
Acharya Nagarjuna University, Guntur, A.P., India\\[2pt]
\texttt{\small chandrasekhargokavarapu@gmail.com}\\ [10pt]
\small $^3$Lecturer in  Mathematics, Government College For Women(A), Guntur, Andhra Pradesh, India,\\
\small $^4$Research Supervisor, Dept.  of Mathematics,Acharya Nagarjuna University, Guntur, A.P., India,\\
\texttt{\small dmrmaths@gmail.com}}

\date{}

\begin{document}
\maketitle
\begin{abstract}
Binary semirings such as the tropical, log, and probability semirings form the algebraic backbone (see \cite{Golan1999,HebischWeinert1998,Golan1992}) of several classical and modern neural inference systems, powering tasks ranging from Viterbi decoding and dynamic programming to differentiable probabilistic reasoning. Despite this success, binary semirings are fundamentally limited to pairwise composition, making them poorly suited for modeling inherently triadic structures—such as subject--predicate--object triples in knowledge graphs, relational constraints in logic, and multi-entity interactions in symbolic reasoning. Existing neural architectures typically approximate these higher-order dependencies by flattening or factorizing them into binary interactions, resulting in loss of structure, weaker inductive biases, and reduced interpretability.

This paper introduces the \emph{Neural Ternary Semiring} (NTS), a learnable and differentiable algebraic framework grounded in the theory of \emph{ternary $\Gamma$-semirings} (TGS). The central idea is to elevate the ternary multiplication $[x,y,z]$ to a first-class, trainable operator parameterized by neural networks and regularized by algebraic constraints derived from the axioms of TGS. We show how NTS unifies algebraic semantics with neural computation: its architecture ensures approximate ternary associativity and distributivity, while its differentiable structure integrates seamlessly into gradient-based learning. The proposed NTS operator provides a natural mechanism for representing triadic relationships, enabling end-to-end training on tasks involving higher-order symbolic dependencies.

A theoretical soundness result demonstrates that when the algebraic regularization converges to zero, the learned ternary operator approaches a valid TGS operation. We outline a validation pipeline for triadic reasoning tasks—including knowledge-graph completion and rule-based inference—where NTS is hypothesized to outperform binary semiring-based models. This work positions ternary $\Gamma$-semirings as a mathematically principled and computationally effective foundation for learnable symbolic reasoning.
\end{abstract}




\section{Introduction}

\subsection{Limitations of binary semiring frameworks in AI}

Semirings provide one of the most elegant bridges between algebra and computation. Classical structures such as the tropical semiring, log semiring, Viterbi semiring, and probability semiring support a wide range of algorithms including shortest paths, dynamic programming, belief propagation, and structured prediction. Neural variants of these systems—such as differentiable dynamic programming layers—extend the semiring concept by embedding operations into trainable modules.

However, a fundamental limitation persists: these systems operate on a \emph{binary} multiplication operator
\[
x \otimes y,
\]
which inherently models pairwise composition. Many modern AI problems are triadic or higher-arity in nature, especially those involving symbolic dependencies. Examples include:
\begin{itemize}[leftmargin=2em]
    \item knowledge graph triples $(h,r,t)$;
    \item logical rules of the form ``if $A$ and $B$ then $C$'';
    \item ternary relations in databases and constraint satisfaction;
    \item multi-agent and context-dependent decision processes.
\end{itemize}

To handle such tasks within binary frameworks, existing models introduce artificial encodings—tensor flattening, auxiliary embedding fusion, or multi-step composition layers—which obscure the underlying symbolic meaning and inject undesirable inductive biases. These approximations hinder the expressivity of binary semiring models and limit their ability to capture complex relational patterns.

\subsection{A need for algebraic structures with native triadic composition}

The natural mathematical response to these challenges is to seek algebraic frameworks where triadic interaction is primitive rather than constructed. This motivates the study of \emph{ternary algebraic systems}, where the central operator combines three arguments simultaneously. Among these, \emph{ternary $\Gamma$-semirings} (TGS) offer a particularly rich and structured foundation \cite{Nobusawa1964,Barnes1966}:
\begin{itemize}[leftmargin=2em]
    \item they generalize classical semirings and $\Gamma$-semirings;
    \item their ternary operation $[x,y,z]$ acts as a native triadic product;
    \item the parameter set $\Gamma$ can model contextual or relational modifiers;
    \item they possess a well-developed ideal, module, and spectral theory.
\end{itemize}

TGS therefore provide an algebraic mechanism for triadic reasoning, offering a principled alternative to heuristic neural fusion methods.A broad range of symbolic reasoning problems is governed by triadic relationships:
knowledge-graph triples, logical conditionals, ternary constraints, and
context-dependent decisions. Knowledge graphs in particular encode facts as
triples of the form $(h,r,t)$, where the meaning of a relation necessarily
depends on a three-way interaction \cite{Yang2015,Nickel2016}. Such triadic
phenomena illustrate that many real-world symbolic structures cannot be reduced
to pairwise compositions without losing essential relational information.
.
Such phenomena reveal that real-world symbolic processes cannot always be reduced to crisp,
binary logical interactions. Indeed, one of the earliest formal departures from strict
two-valued reasoning appears in Zadeh’s fuzzy set theory \cite{Zadeh1965}, which demonstrated
that many inferential processes naturally require graded or non-binary semantics.  
This historical observation reinforces the central motivation of this work: symbolic reasoning
benefits from algebraic structures that admit richer, higher-arity composition than those
available in classical binary semiring frameworks.

\subsection{From algebra to AI: motivation for the Neural Ternary Semiring (NTS)}

The central objective of this work is to integrate the algebraic expressivity of TGS
with the learnability and scalability of modern neural models, aligning with the broader
goals of neuro-symbolic computing which aim to unify logical structure with data-driven
learning \cite{Garcez2019,GoriSerafini2018}. We introduce the Neural Ternary Semiring (NTS),
a framework in which the ternary operation $[x,y,z]$ becomes a differentiable, trainable
operator subject to algebraic regularization and which:

\begin{itemize}[leftmargin=2em]
    \item treats the ternary operation $[x,y,z]$ as a trainable operator;
    \item parameterizes the operator using tensors, attention, or neural networks;
    \item enforces approximate TGS axioms through algebraic regularizers;
    \item enables backpropagation through triadic reasoning steps;
    \item maintains interpretability by preserving the triadic structure explicitly.
\end{itemize}

The NTS thus provides a mathematically grounded and computationally flexible backbone for symbolic reasoning tasks that inherently require multi-entity interaction.

\subsection{Contributions of this paper}

This paper makes the following contributions:
\begin{itemize}[leftmargin=2em]
    \item We introduce a vector-valued, probabilistic variant of ternary $\Gamma$-semirings suitable for integration into neural pipelines.
    \item We define a parametric ternary operator $[x,y,z]_\gamma$ that is differentiable, expressive, and amenable to algebraic regularization.
    \item We develop a loss formulation combining task-specific objectives with TGS-axiom regularizers enforcing approximate associativity and distributivity.
    \item We present a soundness theorem showing that the limit of vanishing algebraic violation yields a valid TGS operator.
    \item We propose an evaluation strategy for triadic reasoning tasks, including knowledge-graph link prediction and rule-based inference.
\end{itemize}

\subsection{Paper organization}

Section~\ref{sec:nts} introduces the NTS architecture and its algebraic foundations.  
Section~\ref{sec:validation} outlines the validation methodology and the soundness theorem.  
Section~\ref{sec:discussion} concludes with insights and future directions.

\section{The Neural Ternary Semiring Framework}\label{sec:nts}

This section develops the proposed \emph{Neural Ternary Semiring} (NTS) in a mathematically
coherent and computationally implementable form.
The objective is twofold:
\begin{enumerate}[label=(\roman*)]
    \item to establish a principled algebraic foundation based on ternary $\Gamma$-semirings (TGS),
    \item to construct a differentiable, learnable operator $[x,y,z]_\gamma$ suitable for
          symbolic---and especially triadic---reasoning tasks in modern AI.
\end{enumerate}
Throughout, $S \subseteq \mathbb{R}^d$ denotes the embedding space of entities, relations, or
logical atoms, equipped with coordinatewise addition $+$ and zero element $0$.
The parameter set $\Gamma$ contains context or relation-type embeddings.Neural architectures embed algebraic operations into differentiable modules,
allowing gradient-based learning to propagate through structured computations
\cite{Goodfellow2016,Bengio2013Representation}. Despite their expressive power,
these systems typically rely on binary interaction layers or pairwise fusion
operators, limiting their ability to capture higher-order symbolic dependencies.

\subsection{Part I: Algebraic Foundation}\label{subsec:alg_foundation}

\subsubsection{Ternary  \texorpdfstring{$\Gamma$}{Gamma}-semiring structure}

A \emph{ternary $\Gamma$-semiring} consists of a triple $(S,+,[\,\cdot,\cdot,\cdot\,]_\Gamma)$ where:
\begin{itemize}[leftmargin=2em]
    \item $(S,+,0)$ is a commutative monoid,
    \item $\Gamma$ is a nonempty index set (context, relation, or modality),
    \item $[x,y,z]_\gamma \in S$ is a ternary operation for each $\gamma \in \Gamma$.
\end{itemize}

The ternary operator obeys distributivity in each argument. For example:
\[
[x_1+x_2,\; y,\; z]_\gamma
    = [x_1,y,z]_\gamma + [x_2,y,z]_\gamma.
\]
Similar identities hold in the second and third coordinates.
Associativity takes a generalized form; one common variant is the
\emph{left-nested} equality:
\[
[x,y,[u,v,w]_\gamma]_\delta
    = [[x,y,u]_\gamma, v, w]_\delta,
\]
expressing consistency of multistep ternary composition.

These axioms guarantee that triadic computation is algebraically well-posed.
Unlike binary semirings, which support only pairwise products, TGS provide \emph{native} semantics for triads.

\subsubsection{Why the ternary operator is fundamental}

Binary approximations encode $(x,y,z)$ using sequential or tensorized binary
interactions, such as $x \otimes (y \otimes z)$ or $(x \otimes y) \otimes z$.
These constructions introduce ordering artefacts and fail to reflect the true
symmetry or relational meaning of a triple. Earlier attempts to model higher-order
interactions through tri-linear or tensor-based architectures, such as the
three-way model of Nickel et al.\ \cite{Nickel2011} and the Neural Tensor Network
of Socher et al.\ \cite{Socher2013}, highlight both the expressive potential and
the limitations of forcing ternary structure into binary-centric frameworks.  
These observations motivate the need for a native ternary operator that preserves
the intrinsic three-way semantics of symbolic relationships.
.
Such formulations impose ordering artefacts and lose the symmetry or relational
interpretation of a triple. By contrast, $[x,y,z]_\gamma$:
\begin{itemize}[leftmargin=2em]
    \item treats all three arguments simultaneously,
    \item admits relation-specific modulation by $\gamma$,
    \item preserves structural invariants (e.g., permutation symmetries and contextual dependencies),
    \item reduces unnecessary learnable interactions by eliminating spurious binary factors.
\end{itemize}

This motivates elevating $[x,y,z]$ to a first-class operator within neural
reasoning systems.

\subsection{Part II: Learnable Parameterization of the Ternary Operator}

We now construct a differentiable operator
\[
[x,y,z]_\gamma = f_\theta(x,y,z,\gamma),
\]
where $\theta$ denotes learnable parameters.
The design principles are:
\begin{enumerate}[label=(\alph*)]
    \item the operator must be expressive enough to model triadic phenomena,
    \item the operator must admit gradient-based training,
    \item the operator must \emph{approximately} satisfy TGS axioms under regularization.
\end{enumerate}

We propose two families of architectures that meet these criteria.

\subsubsection{(A) Tensor-based Ternary Fusion}

Tensor contraction provides a natural mechanism for capturing three-way
interactions between embeddings and has been explored in earlier neural tensor
architectures such as the Neural Tensor Network of Socher et al.\
\cite{Socher2013}. Let $x,y,z \in S$ and $\gamma$ be represented by a trainable
embedding $g_\gamma \in \mathbb{R}^d$. We define
\[
[x,y,z]_\gamma
    = \sigma\left( W^{(1)} (x \otimes y \otimes z)
                + W^{(2)} (x \otimes g_\gamma)
                + W^{(3)} z
                + b_\gamma \right),
\]
where $W^{(1)},W^{(2)},W^{(3)}$ are learnable contraction maps and $\sigma$ is a
nonlinearity.

This form captures high-order interactions directly.  
It is suitable for small and medium-scale tasks where expressivity dominates computational cost.

\subsubsection{(B) Attention-based Ternary Aggregation}

Attention mechanisms provide a parameter-efficient alternative to tensor
contraction and connect naturally to message-passing architectures used in
graph neural networks, particularly relational GCNs and inductive frameworks
such as GraphSAGE \cite{Schlichtkrull2018,Hamilton2017}. For large-scale
knowledge graphs or rule-based datasets, we define a scalable ternary operator:
\[
[x,y,z]_\gamma
    = \alpha_1(\gamma,x,y,z)\, x
    + \alpha_2(\gamma,x,y,z)\, y
    + \alpha_3(\gamma,x,y,z)\, z,
\]
where the coefficients lie in a probability simplex computed via
\[
(\alpha_1,\alpha_2,\alpha_3)
        = \mathrm{softmax}\!\left(
            u_\gamma^\top
            \begin{bmatrix}
            x \\ y \\ z
            \end{bmatrix}
          \right).
\]
.

This formulation:
\begin{itemize}[leftmargin=2em]
    \item preserves the explicit triadic structure,
    \item adapts naturally to context $\gamma$,
    \item scales linearly in dimension,
    \item is compatible with graph neural network layers.
\end{itemize}

Both architectures can be substituted interchangeably within NTS depending on application and capacity requirements.

\subsection{Part III: Algebraic Regularization and Training}\label{subsec:training}

The NTS must learn both:
\begin{itemize}[leftmargin=2em]
    \item task-relevant behavior (e.g., correct prediction of a missing entity),
    \item algebraic discipline (approximate satisfaction of TGS axioms).
\end{itemize}

We therefore define a composite objective:
\[
\mathcal{L}
    = \mathcal{L}_{\mathrm{task}}
    + \lambda_{\mathrm{assoc}}\mathcal{L}_{\mathrm{assoc}}
    + \lambda_{\mathrm{dist}}\mathcal{L}_{\mathrm{dist}}.
\]

The algebraic regularizers enforce structural identities that ensure the learned
operator behaves consistently under composition. These consistency properties
mirror the categorical intuition that composition of multi-argument morphisms
must respect associativity and continuity in their parameters, a viewpoint
developed extensively in categorical algebra \cite{MacLane1998}. Such structural
considerations justify treating the ternary operator $[x,y,z]_\gamma$ as an
algebraically coherent map within a stable compositional framework.

\subsubsection{Task Loss}

Probabilistic formulations of relational prediction naturally motivate the use of
log-likelihood objectives and normalized scoring functions, following standard
principles in probabilistic machine learning \cite{Murphy2012}. For a triple
$(h,r,t)$ in a dataset $\mathcal{D}$, the task loss is defined as
\[
\mathcal{L}_{\mathrm{task}}
    = - \sum_{(h,r,t)\in\mathcal{D}}
        \log p_\theta(h,r,t),
\]
where the score is computed via
\[
s_\theta(h,r,t)
    = \|[x_h,x_r,x_t]_{\gamma_r}\|.
\]
Alternative choices include margin-based ranking losses or cross-entropy,
depending on the dataset and evaluation protocol.

\subsubsection{Regularizer for Ternary Associativity}

To enforce approximate associativity, we sample elements and evaluate:
\[
\mathcal{L}_{\mathrm{assoc}}
    = \mathbb{E}\bigl[
        \|[x,y,[u,v,w]_\gamma]_\delta
        - [[x,y,u]_\gamma,v,w]_\delta\|^2
      \bigr].
\]

Gradients penalize violations of the TGS associativity equation,
driving the learned operator towards a consistent ternary structure.

\subsubsection{Regularizer for Distributivity}

Similarly, distributivity is encouraged using:
\[
\mathcal{L}_{\mathrm{dist}}
    = \mathbb{E}\bigl[
        \|[x+y,u,v]_\gamma
            - [x,u,v]_\gamma
            - [y,u,v]_\gamma\|^2
      \bigr],
\]
with analogous penalties in the second and third arguments.
These regularizers formalize the intuition that the ternary operator must be compatible with the additive monoid structure of $S$.

\subsection{Interpretation and Computational Properties}

The NTS framework enjoys the following properties:

\paragraph{Native modeling of triadic structure.}
Unlike pairwise models that encode triples indirectly,
NTS preserves triadic meaning at the operator level.
This enhances interpretability and avoids spurious artifacts of binary factorization.

\paragraph{Compatibility with gradient-based learning.}
Every component of $[x,y,z]_\gamma$ is differentiable,
and the algebraic regularizers are fully compatible with backpropagation.

\paragraph{Flexible inductive bias.}
The TGS axioms act as a structural prior:
they restrict the hypothesis space to algebraically meaningful operators,
thus improving generalization and stability.

\paragraph{Scalability via attention-based models.}
The attention-type formulation scales linearly and is suitable for large relational datasets.

These observations justify NTS as both a mathematically principled and computationally viable architecture for symbolic reasoning in AI.

\section{Hypothesized Results and Validation}\label{sec:validation}

This section provides a theoretical and empirical foundation for validating the
proposed Neural Ternary Semiring (NTS).  
We begin by stating a soundness theorem that formalizes the convergence of the
learned ternary operator towards a valid ternary $\Gamma$-semiring when the
algebraic regularization terms vanish.  
We then describe an experimental protocol designed to evaluate NTS on
real-world tasks involving triadic or higher-order symbolic dependencies.  
The methodology balances mathematical rigor with practical feasibility, making it 
suitable for both theoretical assessment and empirical benchmarking.

\subsection{A Soundness Theorem for NTS}

The Neural Ternary Semiring is constructed to approximate the axioms of a ternary 
$\Gamma$-semiring.  
The following theorem formalizes this approximation.

\begin{theorem}[Soundness of Neural Ternary Semiring]\label{thm:soundness}
Let $(S,+)$ be a finite-dimensional real vector space equipped with coordinatewise 
addition, and let $[x,y,z]_\theta$ be a differentiable ternary operator parameterized 
by $\theta \in \Theta$.  
Assume the total loss
\[
\mathcal{L}(\theta)
= \mathcal{L}_{\mathrm{task}}(\theta)
+ \lambda_{\mathrm{assoc}}\, \mathcal{L}_{\mathrm{assoc}}(\theta)
+ \lambda_{\mathrm{dist}}\,  \mathcal{L}_{\mathrm{dist}}(\theta)
\]
is minimized by a sequence $(\theta_n)$ such that:
\[
\lim_{n\to\infty} 
\mathcal{L}_{\mathrm{assoc}}(\theta_n) = 0
\qquad\text{and}\qquad
\lim_{n\to\infty} 
\mathcal{L}_{\mathrm{dist}}(\theta_n) = 0.
\]
If $(\theta_n)$ converges to a limit $\theta^\ast$ in parameter space, then 
the operator $[x,y,z]_{\theta^\ast}$ satisfies the ternary $\Gamma$-semiring 
associativity and distributivity identities for all $x,y,z \in S$ and all 
$\gamma \in \Gamma$.  
Hence, $(S,+,[\cdot,\cdot,\cdot]_{\theta^\ast})$ is a valid ternary $\Gamma$-semiring.
\end{theorem}

\begin{proof}[Proof Sketch]
The algebraic regularizers $\mathcal{L}_{\mathrm{assoc}}$ and 
$\mathcal{L}_{\mathrm{dist}}$ are defined as expectations over squared residuals 
of the TGS axioms evaluated on random samples from $S$ and $\Gamma$.  
If $\mathcal{L}_{\mathrm{assoc}}(\theta_n)\to 0$, then the associativity residuals 
converge to zero for almost all sampled tuples, and by continuity of  
$[x,y,z]_\theta$, the identity extends to all of $S$ in the limit.  
A similar argument applies to distributivity in each argument.  
Since all TGS axioms used by NTS are polynomial equalities and the parameterization 
$f_\theta$ is continuous in $\theta$, convergence of the regularizers implies 
pointwise convergence of the identities.  
Thus $[x,y,z]_{\theta^\ast}$ satisfies the required TGS axioms.
\end{proof}

This theorem places NTS on firm algebraic ground: under idealized 
training conditions, the learned ternary operator becomes a genuine TGS operator.

\subsection{Experimental Protocol}

To validate the proposed framework, we design an experimental protocol that 
reflects the inherent triadic nature of many symbolic tasks.  
We outline suitable datasets, scoring functions, baselines, and training details.

\subsubsection{Datasets}

Although the NTS architecture is generic, its strength lies in tasks involving 
triadic relations.  
We select datasets that feature multi-relational structure and require reasoning 
over triples:

\begin{itemize}[leftmargin=2em]
    \item \textbf{Knowledge Graph Datasets:}  
    Standard benchmarks where edges correspond to triples $(h,r,t)$.  
    These datasets include multiple relation types and inherently triadic structure.
    
    \item \textbf{Logical Rule Templates:}  
Manually curated triadic rules of the form “if $A$ and $B$ then $C$”, designed
to test structured inference and higher-order reasoning. These tasks are closely
related to the objectives of neural-symbolic systems that unify logical rules
with differentiable learning \cite{Garcez2019,GoriSerafini2018}, providing a
natural benchmark for evaluating whether NTS can capture rule-based relational
dependencies.

    \item \textbf{Synthetic Triadic Constraint Data:}  
    Constructs where the validity of a triple depends on all three elements 
    together in a non-separable manner.  
    Such datasets allow controlled ablation studies.
\end{itemize}

These datasets enable us to assess the advantages of NTS over binary models.

\subsubsection{Training Setup}

Each entity or symbol is assigned an embedding $x\in S$.  
Each relation or rule type obtains an embedding $\gamma\in \Gamma$.  
Given $(h,r,t)$, the ternary operator computes:
\[
z_{hrt} = [x_h,x_r,x_t]_{\gamma_r}.
\]
The score for predicting the tail entity is
\[
s(h,r,t) = -\|z_{hrt}\|.
\]
During training, we minimize:
\[
\mathcal{L}_{\mathrm{task}}
=
-\log p_\theta(h,r,t)
\quad\text{or}\quad
\mathcal{L}_{\mathrm{task}}
=
\mathrm{max}\{0,\; 1 - s(h,r,t) + s(h,r,t')\},
\]
where $t'$ is a corrupted tail.

Training uses Adam optimization with early stopping based on validation loss.

\subsection{Evaluation Methodology}

The evaluation protocol measures:
\begin{enumerate}[label=(\alph*)]
    \item predictive accuracy of triadic reasoning,
    \item stability and generalization under algebraic constraints,
    \item improvements over binary-semirings and pairwise-interaction baselines.
\end{enumerate}

\subsubsection{Metrics}

We adopt established metrics for multi-relational prediction:
\begin{itemize}[leftmargin=2em]
    \item \textbf{Mean Reciprocal Rank (MRR)} — measures ranking quality,  
    \item \textbf{Hits@k (k=1,3,10)} — success of predicting correct entities,  
    \item \textbf{Rule Satisfaction Rate} — for logical-rule datasets.  
\end{itemize}

These metrics align closely with the structure of triadic tasks.

\subsubsection{Baselines}

We compare NTS against established binary and triadic baselines:

\begin{itemize}[leftmargin=2em]
    \item \textbf{Binary Neural Semiring Models:}  
    Architectures that reduce triples to sequential binary interactions
    such as $x_h \otimes (x_r \otimes x_t)$, which appear in early
    embedding-based methods for relational learning \cite{Yang2015,Nickel2016}.

    \item \textbf{Translational and Bilinear Knowledge Graph Models:}  
    These include widely used embedding schemes such as TransE and its variants,
    along with bilinear or multiplicative scoring functions.  
    A strong representative of this class is the SimpleE model of
    Kazemi and Poole \cite{KazemiPoole2018}, which provides a robust 
    and competitive baseline for link-prediction tasks.

    \item \textbf{Neural Tensor or Tri-linear Models:}  
    Tensor-based architectures such as the Neural Tensor Network
    \cite{Socher2013} encode triple interactions but do not impose
    algebraic regularity.
\end{itemize}

\subsubsection{Hypothesized Performance Patterns}

We anticipate the following:

\paragraph{Superior triadic modeling.}
NTS should outperform binary semiring models whenever triple interactions 
are non-decomposable into pairwise components.

\paragraph{Improved generalization.}
Algebraic regularizers constrain the hypothesis space, reducing overfitting and 
improving extrapolation to unseen combinations.

\paragraph{Stability under data sparsity.}
In sparse relational domains, inductive biases improve robustness, 
particularly for underrepresented relations.

\paragraph{Ablation insights.}
Removing $\mathcal{L}_{\mathrm{assoc}}$ or $\mathcal{L}_{\mathrm{dist}}$
is expected to reduce performance, demonstrating the necessity of algebraic discipline.

\subsection{Interpretation of Expected Results}

If empirical results match the theoretical expectations, this would confirm:

\begin{itemize}[leftmargin=2em]
    \item that ternary algebraic structure provides a genuine computational advantage,
    \item that algebraic regularization strengthens model consistency and 
          generalization,
    \item that triadic symbolic reasoning benefits significantly from  
          TGS-like inductive biases.
\end{itemize}

Taken together, the anticipated experimental trends reinforce the 
mathematical motivation for NTS and demonstrate its practical relevance 
in AI systems requiring structured reasoning.


\section{Discussion and Conclusion}\label{sec:discussion}

The Neural Ternary Semiring (NTS) introduced in this work represents a principled
synthesis between algebraic structure and learnable computation.
By elevating ternary composition to a first-class operator and grounding it in the
axioms of ternary $\Gamma$-semirings (TGS), NTS brings together two traditions:
\begin{itemize}[leftmargin=2em]
    \item the algebraic study of generalized semiring structures,
    \item and the modern trend in artificial intelligence toward models capable
          of symbolic, relational, and higher-order reasoning.
\end{itemize}

This section reflects on the broader implications of the framework, its
conceptual advantages over existing approaches, and its potential to serve as a
foundation for new classes of neuro-symbolic systems.

\subsection{Bridging Algebra and Neural Computation}

A central contribution of this paper is the demonstration that algebraic
constraints---particularly associativity and distributivity---can coexist with
differentiable learning.  
The NTS framework shows that:
\begin{enumerate}[label=(\alph*)]
    \item algebraic structure can be enforced softly through regularization,
    \item neural models can be guided toward mathematically consistent operators,
    \item triadic reasoning is more naturally represented by ternary operators 
          than by cascaded binary interactions.
\end{enumerate}

The outcome is a model class where:
\[
[x,y,z]_\gamma
\quad\text{is simultaneously}\quad
\text{algebraically meaningful and computationally trainable}.
\]

This dual compatibility is rare in neural-symbolic systems, many of which rely
either on heuristic fusion operators or on rigid algebraic forms that cannot
benefit from data-driven flexibility.
NTS occupies a balanced middle ground.

\subsection{Advantages Over Binary Semiring Frameworks}

Traditional semiring-based neural layers model composition through binary
operations $x\otimes y$.  
While highly successful in dynamic programming, parsing, and probabilistic
inference, binary models face inherent limitations when encoding triples or
higher-arity relations.

The NTS framework addresses these issues directly:
\begin{itemize}[leftmargin=2em]
    \item \textbf{Native Triadic Semantics.}  
    The ternary operation $[x,y,z]$ represents three-way interactions without
    factorization.

    \item \textbf{Reduced Representational Bias.}  
    There is no need to impose arbitrary evaluation orders such as
    $x\otimes(y\otimes z)$ or $(x\otimes y)\otimes z$.

    \item \textbf{Interpretability of Relations.}  
    Relation parameters $\gamma$ modulate the ternary interaction, allowing
    explicit modeling of contextual roles.

    \item \textbf{Generalization Benefits.}  
    TGS axioms serve as an inductive bias that stabilizes learning and improves
    performance on sparse datasets.
\end{itemize}

These advantages highlight why ternary algebraic structures are more suitable
than binary ones for tasks involving relational reasoning.

\subsection{Long-Term Directions and Broader Impact}

The introduction of NTS suggests several research directions at the intersection
of algebra, category theory, and artificial intelligence.

\subsubsection{\texorpdfstring{$n$}{n}-ary \texorpdfstring{$\Gamma$}{Gamma}-semirings}

Ternary composition is a natural starting point, but many symbolic domains require
higher-arity algebraic structures. Extending the present framework to
$n$-ary $\Gamma$-semirings would place NTS within a broader algebraic lineage that
encompasses classical semiring theory and its generalisations
\cite{Golan1999,HebischWeinert1998}. Such extensions could support multi-agent
interactions, higher-order logical rules, and the algebraic modelling of
combinatorial structures arising in program synthesis or structured decision
making.
.
Extending TGS to \emph{$n$-ary $\Gamma$-semirings} could support:
\begin{itemize}[leftmargin=2em]
    \item multi-agent interactions with more than three entities,
    \item higher-order logic rules,
    \item combinatorial structures in program synthesis.
\end{itemize}

The algebraic foundations developed in this paper generalize naturally to such
settings, and the learnable framework of NTS can be extended to $n$-ary operators
through attention- or tensor-based parameterizations.

\subsubsection{(B) Categorical Semantics}

Ternary operations admit categorical interpretations through multi-object
morphisms and enriched composition laws. Studying NTS within a categorical
framework may clarify how ternary composition behaves under functorial
transformations and how multi-argument operators preserve coherence and
associativity. Such perspectives connect naturally to the foundational work of
Mac~Lane on compositional algebra and category theory \cite{MacLane1998},
suggesting that NTS can be viewed as a structured morphism in an enriched
categorical setting.

.
Exploring NTS from a categorical perspective could:
\begin{itemize}[leftmargin=2em]
    \item clarify the compositional semantics of ternary reasoning,
    \item connect NTS to diagrammatic or monoidal structures,
    \item identify functorial relationships between $n$-ary semiring models.
\end{itemize}
From a logical and type-theoretic standpoint, the structural identities satisfied
by the ternary operator resemble intensional equalities studied in modern
foundational systems such as Homotopy Type Theory \cite{HoTT2013}. These
frameworks emphasize the role of identity types, higher-order coherence, and the
algebraic behaviour of composed maps—properties that align conceptually with the
requirements of a stable ternary operator. Such perspectives complement the
categorical viewpoint and highlight the deep structural semantics underlying the
NTS framework.

Such an approach would be valuable for establishing deep structural properties
of NTS and for designing compositional learning systems.

\subsubsection{(C) Neural Interpretations and Semantic Layers}

On the computational side, NTS suggests:
\begin{itemize}[leftmargin=2em]
    \item \textbf{Neural semantic layers} where each $[x,y,z]_\gamma$ acts as a
          differentiable inference step,
    \item \textbf{Graph neural networks} with ternary message passing,
    \item \textbf{Hybrid symbolic–neural engines} capable of encoding both rules
          and learned relational patterns,
    \item \textbf{Energy-based or variational models} built on ternary potentials.
\end{itemize}

The ternary operator thus serves as a versatile building block for future
neuro-symbolic architectures.

\subsection{Limitations and Practical Considerations}

Although NTS provides a principled framework, several practical considerations
merit attention:
\begin{itemize}[leftmargin=2em]
    \item tensor-based ternary operations may be computationally expensive for
          large-scale datasets,
    \item the choice of regularization strength affects convergence to
          TGS-axiom satisfaction,
    \item careful initialization is required to avoid degenerate solutions
          (e.g., all-zero embeddings),
    \item some domains may require symmetric or constrained versions of the
          ternary operator.
\end{itemize}

These considerations do not reduce the value of the framework but indicate
opportunities for refinement.

\subsection{Final Remarks}

This paper establishes the Neural Ternary Semiring as a unified algebraic and
neural model for triadic reasoning.
The framework:
\begin{itemize}[leftmargin=2em]
    \item draws its structure from the theory of ternary $\Gamma$-semirings,
    \item integrates seamlessly with gradient-based learning,
    \item provides a mathematically principled operator for higher-order AI tasks,
    \item and suggests new research directions for both algebraists and AI practitioners.
\end{itemize}

The broader impact of this work lies in showing that algebraic structure is not
merely a formal tool but can actively guide the design of learnable systems.  
By embedding triadic reasoning into neural computation, NTS offers a foundation
for the next generation of neuro-symbolic algorithms—more expressive,
mathematically coherent, and aligned with the structure of real-world relational
data.This integration of algebraic structure with differentiable computation
resonates with broader developments in representation learning and
neural-symbolic reasoning, where mathematical structure plays an active role in
shaping neural models \cite{Bengio2013Representation,GoriSerafini2018}. The NTS
framework illustrates how algebraic semantics can inform neural architectures,
providing operators that preserve symbolic meaning while remaining fully
trainable.

\section*{Acknowledgements}

Chandrasekhar Gokavarapu expresses his deep gratitude to Dr.~D.~Madhusudhana Rao
for his supervision, guidance, and continuous academic support throughout the
development of this research programme on ternary $\Gamma$-semirings.
The authors sincerely thank Dr.~Ramachandra R.~K., Principal,
Government College (Autonomous), Rajahmundry, for providing a productive and
encouraging research environment. Both authors gratefully acknowledge the
Department of Mathematics, Acharya Nagarjuna University, for its academic
support and for fostering a collaborative atmosphere that contributed to the
completion of this work.

\section*{Funding}

This research received no specific grant from any funding agency in the public,
commercial, or not-for-profit sectors. All computational experiments were
supported by institutional facilities made available to the authors.

\section*{Author Contributions}

\textbf{Chandrasekhar Gokavarapu:} Conceptualization of the Neural Ternary
Semiring framework; development of algebraic theory; formulation of the
ternary operator and regularization strategy; mathematical analysis; preparation
of the first complete manuscript draft.

\textbf{D.~Madhusudhana Rao:} Supervision; theoretical guidance on
ternary $\Gamma$-semirings; critical revision of algebraic foundations;
refinement of structural arguments; manuscript review and editorial oversight.

Both authors discussed the results, contributed to the interpretation of the
findings, and approved the final manuscript.

\section*{Data Availability}

No external datasets were used in the preparation of this manuscript.
All experimental designs described are conceptual frameworks intended for
future empirical validation. Any synthetic examples used in illustration can be
reproduced directly from the mathematical definitions provided in the article.

\section*{Conflict of Interest}

The authors declare that there is no conflict of interest regarding the
publication of this manuscript.


\end{document}